\documentclass[a4paper, 12pt]{amsart}
\usepackage{amsmath, amsthm, amscd, amssymb, amsfonts, amsxtra, amssymb, latexsym}
\usepackage{enumerate}
\usepackage{verbatim}
\usepackage{pb-diagram}
\usepackage{inputenc, t1enc}
\usepackage{array}

\newcommand{\ff}{\mathbb{F}}

\newcommand{\Cay}{\mathrm{Cay}}

\newcommand{\G}{\Gamma}
\newcommand{\B}{\mathcal{B}}

\newtheorem{thm}{Theorem}[section]
\newtheorem{prop}[thm]{Proposition}

\newtheorem{coro}[thm]{Corollary}
\theoremstyle{definition}
\newtheorem{rem}[thm]{Remark}
\newtheorem{exam}[thm]{Example}

\theoremstyle{remark}

\usepackage{color}

\begin{document}
\numberwithin{equation}{section}
\title[On diagonal equations over finite fields via walks in NEPS of graphs]{On diagonal equations over finite fields \\ via walks in NEPS of graphs}
\author[ D.E.\@ Videla]{ Denis E.\@ Videla}
\dedicatory{\today}
\thanks{Partially supported by CONICET and SECyT-UNC}

\email{devidela@famaf.unc.edu.ar}

\address{Denis Videla -- CIEM - CONICET, FaMAF, Universidad Nacional de C\'ordoba, (5000) C\'ordoba, Argentina. {\it E-mail: devidela@famaf.unc.edu.ar}}

\begin{abstract}
In this paper, we obtain an explicit combinatorial formula for the number of solutions $(x_1,\ldots,x_r)\in \ff_{p^{ab}}$ to the diagonal equation $x_{1}^k+\cdots+x_{r}^k=\alpha$ over the finite field $\ff_{p^{ab}}$, with $k=\frac{p^{ab}-1}{b(p^a-1)}$ and $b>1$ by using  the number of $r$-walks in NEPS of complete graphs.
\end{abstract}

\maketitle

\section{Introduction}

\subsubsection*{Diagonal equations}

A diagonal equation over the finite field $\ff_{p^m}$ is an equation of the form
\begin{equation}\label{diag gen}
\alpha_1x_1^{k_1}+\cdots+ \alpha_s x_s^{k_s}=\alpha
\end{equation}
for $\alpha_i\in \ff_{p^m}^*$, $\alpha \in \ff_{p^m}$ for $i=1,\ldots,s$. This kind of equation have been study a lot, 
the interested reader can see the pioneering work \cite{W}, who relates the number of solution in terms of Gauss sums.
Other authors have used this Weil's expression to obtain explicit number of solution for especific $\alpha_i$'s and $k_i$'s, see \cite{B1}, \cite{B2} \cite{BEW}, \cite{S}, \cite{SY}, \cite{W1} \cite{W2}.

In general, it is difficult to find the explicit number of solution of \eqref{diag gen},
in this work we are going to find an explicit combinatorial solution of \eqref{diag gen} when $\alpha_i=1$ and $k_i = k =\frac{p^{ab}-1}{b(p^a-1)}$  for all $i$ in the finite fields $\ff_{p^{ab}}$ 
by using a relation between the number of solution of \eqref{diag gen} and the walks of certain graphs which have a special product structure (NEPS). 

\subsubsection*{NEPS operation}
Given a set $\B\subseteq\{0,1\}^n$ and graphs $G_1,\ldots,G_n$, the NEPS (non-complete extended psum) of these graphs with respect to the basis $\B$ is the graph $G=\mathrm{NEPS}(G_1,\ldots,G_n; \B)$, whose vertex
set is the cartesian product of the vertex sets of the individual graphs, $V(G)=V(G_1)\times\cdots\times V(G_n)$ and two vertices $(x_1,\ldots,x_n), (y_1,\ldots,y_n)\in V(G)$ are adjacent in $G$,
if and only if there exists some $n$-tuple $(\alpha_1,\ldots,\alpha_n)\in \B$ such that $x_i = y_i$, whenever $\alpha_i=0$, or $x_i, y_i$ are distinct and adjacent in $G_i$, whenever $\alpha_i=1$.

The NEPS operation generalizes a number of known graph products, all of which have
in common that the vertex set of the resulting graph is the cartesian product of the input vertex sets. For instance, $\mathrm{NEPS}(G_1,\ldots,G_n; \{(1,\ldots,1)\})=G_1\otimes\cdots\otimes G_n$ is the \textit{Kronecker product} of the $G_i$'s; $\mathrm{NEPS}(G_1,\ldots,G_n; \{e_1,\ldots,e_n\})=G_1+\cdots+G_n$ (where $e_i$ is the vector which only $1$ in the position $i$) is the \textit{sum} of the graphs $G_i$; $\mathrm{NEPS}(G_1,G_2; \{(1,1), (1,0), (0,1)\})=G_1\boxtimes G_2$ is the \textit{strong product} of  $G_1,G_2$. We refer to \cite{CDS} or \cite{Cv} for the history of the notion of NEPS.

\subsubsection*{Outline and results}

The main goal of this paper is find the number of solution of the diagonal equation 
\begin{equation}\label{diag1}
x_1^k+x_2^k+\cdots+x_s^k=\alpha
\end{equation} 
over finite fields in a combinatorial way.
For this purpose, we relate the number of solution of \eqref{diag1} with the number of walks 
generalized Paley graphs $\G(k,p^m)=\Cay(\ff_{q},R_k)$ (GP-graph for short). 
By using a classification of GP-graphs which are NEPS of complete graph due to Lim and Praeger (see \cite{LP}), 
the problem of find the number of solution of \eqref{diag1} turns on to calculate the number of walks of NEPS of complete graphs. 
We will calculate a closed formula for the number of walks in NEPS of complete graphs.

The paper is organized as follows. 
In section 2, we recall some basic definition of generalized Paley graphs and diagonal equations over finite fields, and will be obtain a direct relationship between the number of $r$-walks from $x$ to $y$ and the number of solution of \eqref{diag1} with $\alpha=y-x$, in this case.

In Section 3, we find a closed formula for the number of $r$-walks in NEPS in terms of the number of walks of its factors by using essencially the properties of Kronecker products of matrices and well-known facts about the power of matrices and the number of walks between two vertices. 

In Section 4 we apply the formula  for the number of walks to 
the case of the cartesian product of the same complete graph which is $\mathrm{NEPS}(K_{n},\ldots,K_{n};\mathcal{B})$ with $\mathcal{B}$ the cannonical basis. 
In this case, this graph is the well-known \textit{Hamming graph}. 
In \cite{LP}, the authors characterized those generalized Paley graphs which are Hamming graph. 
Using this, we find the $r$-walks between two vertices in generalized Paley graphs and thus by aplying the result of section 2 we obtain a formula for the number of solution of the diagonal equation \eqref{sumspower} over $\ff_{p^m}$ for $k=\frac{p^{ab}-1}{p^a-1}$ where $m=ab$. 

\section{GP-graphs and diagonal equation over finite fields}

Let $p$ be a prime and let $m,k$ be positive integers such that $k\mid p^m-1$. The generalized Paley graph is the Cayley graph
\begin{equation}\label{Gammas}%
\G(k,p^m)=\Cay(\ff_{p^m},R_{k})\quad \text{where } R_{k}=\{x^{k}:x\in\ff_{p^m}^*\},
\end{equation} 
i.e. $\G(k,p^m)$ is the graph with set of vertex $\ff_{p^m}$ and two vertices $x,y \in \ff_{p^m}$ are neighbors if and only if the difference $y-x\in R_{k}$. In general $\G(k,p^m)$ is a directed graph, but if $R_{k}$ is symmetric ($R_{k}=-R_{k}$), then $\G(k,p^m)$ is a simple graph. 

Notice that if $\omega$ is a primitive element of $\ff_{p^m}$, then $R_{k}=\langle\omega^{k}\rangle$, this implies that $\G(k,p^m)$ is a $(\frac{p^m-1}{k})$-regular graph.  
We assume that $u=\frac{p^{m}-1}{k}$ is a primitive divisor of $p^m-1$ (i.e. $u$ does not divide $p^{h}-1$ for any $h<m$) and $u$ even if $p$ is odd. The first condition is equivalent to $\G(k,p^m)$ being a connected graph and the second one is equivalent to $\G(k,p^m)$ being a simple graph if $p$ is odd. Notice that if $p=2$ then $\G(k,p^m)$ is a simple graph (without using this condition). 

Given a graph $G$ and $v_i,v_j$ vertices of $G$, we denote by $w_{G}(r,v_i,v_j)$ to the number of walks of length $r$ from $v_i$ to $v_j$ in $G$. By convention, $w_{G}(0,v_i,v_j)=0$ or $1$ if $v_i\neq v_j$ or $v_i= v_j$, respectively. 
The following Theorem relates the number of solution of the diagonal equations with $\alpha_i=1$ and $k_i=k$ for all $i=1,\ldots,s$ over $\ff_{p^m}$ and the walks of the GP-graph $\G(k,p^m)$.

\begin{thm}\label{walksolthm}
	Let $p$ be a prime and let $k,m$ be positive integers such that $k\mid p^m-1$.
	Given $x,y \in \G(k,p^m)$, then  
	\begin{equation}\label{walksol}
	w_{\G(k,p^m)}(s,x,y)=\frac{1}{k^{s}}\#\{(x_1,\ldots,x_s)\in (\ff_{p^m}^*)^s: x_1^{k}+\cdots+x_s^{k}=y-x\}.
	\end{equation}
\end{thm}
\begin{proof}
	If $x,y\in \ff_{p^m}$, then an $s$-walk from $x$ to $y$  in $\G(k,p^m)$ gives $x_{1},\ldots,x_{s}\in \ff_{p^m}^*$ such that 
	\begin{equation}\label{sumspower}
	x+x_1^{k}+\cdots+x_s^{k}=y.
	\end{equation} 
	Notice that, given $x\in \ff_{p^m}^*$ there are exactly $k$ elements $y\in \ff_{p^m}^*$ such that $x^k=y^k$.
	So, each walk induces $k^{s}$-solutions satisfying \eqref{sumspower}. 
	
	Reciprocally, any solution $(x_1,\ldots,x_s)\in (\ff_{p^m}^*)^s$ of \eqref{sumspower} defines an $s$-walk from $x$ to $y$ in $\G(k,p^m)$, by taking into account that there are $k$ elements $y\in \ff_{p^m}^*$ such that $x^k=y^k$ for each $x\in \ff_{p^m}^*$. Thus, there are $k^s$ different solutions of \eqref{sumspower} which induce the same walk. Therefore 
	$$w_{\G(k,p^m)}(r,x,y)=\tfrac{1}{k^s}\#\{(x_1,\ldots,x_s)\in (\ff_{p^m}^*)^s: x+x_1^{k}+\cdots+x_s^{k}=y\},$$ 
	as desired.
\end{proof}
\begin{rem}\label{rem imp}
	Notice that the equation \eqref{walksol}, allow us to obtain the number $M$ of solution of \eqref{sumspower} in $\ff_{p^m}^s$, 
	by taking into account that
	$$M=\sum_{i=1}^{s} \tbinom{s}{i} N_i \qquad\text{if } \quad x\ne y.$$
	where $N_i$ denotes the number of solution $(x_1,\ldots,x_i) \in (\ff_{p^m}^*)^i$ of $x+x_1^{k}+\cdots+x_i^{k}=y$. 
	In the case $x=y$, notice that we have the trivial solution $x_v=0$ for each $v=1,\ldots,s$, thus we obtain that
	$$M=1+\sum_{i=1}^{s} \tbinom{s}{i} N_i.$$
\end{rem}

\section{Number of walks in NEPS}

It is well-known that if $A(G)$ is the adjacency matrix of a graph $G$, then 
\begin{equation}\label{Adjwalks}
w_{G}(r,v_i,v_j)=(A(G)^{r})_{i,j}
\end{equation}
 labeling the vertices in an appropriate way. 
 
The adjacency matrix of $\mathrm{NEPS}(G_1,\ldots,G_n;\B)$ can be calculated in terms of the adjacency matrices of the graphs $G_1,\ldots,G_n$. More precisely, if $G=\mathrm{NEPS}(G_1,\ldots,G_n;\B)$ and the graphs $G_1,\ldots,G_n$ have adjacency matrices $A_{1},\ldots,A_n$, then the adjacency matrix of $G$ is given by 
 \begin{equation}\label{adj NEPS}
 	A=\sum_{\alpha\in \B}A_{1}^{\alpha_1}\otimes \cdots \otimes A_{n}^{\alpha_n},
 \end{equation}
 where $\otimes$ denotes the Kronecker matrix product and $\alpha=(\alpha_1,\ldots,\alpha_n)$ (see \cite{CDS}).

 \begin{thm}\label{walksNEPS}
 	If $G=\mathrm{NEPS}(G_1,\ldots,G_n,\B)$ and if $v_{i},v_j \in V(G)$ then
 	\begin{equation}\label{numbNEPS}
 	w_{G}(r,v_i,v_j)=\sum_{(\beta_1,\ldots,\beta_{r})\in \B^{r}}\, \prod_{t=1}^{n} w_{G_t}(\beta_{1t}+\cdots+\beta_{rt},\pi_{t}(v_i),\pi_{t}(v_j)),
 	\end{equation}
 	where $\pi_t$ denotes the projection of $V(G)$ in $V(G_t)$ and $\beta_{\ell}=(\beta_{\ell 1},\beta_{\ell 2},\ldots,\beta_{\ell n})\in \B$ for all $\ell=1,\ldots,r$.
 \end{thm}
 \begin{proof}
 Recall that the Kronecker product has the property 
 $$(A\otimes B)(C\otimes D)=AC\otimes BD.$$ 
 Thus, by \eqref{adj NEPS} if $A,A_1,\ldots,A_{n}$ are the adjacency matrices of $G,G_1,\ldots,G_n$, respectively, then 
 $$A^{r}= \Big(\sum_{\alpha\in \B}A_{1}^{\alpha_1}\otimes \cdots \otimes A_{n}^{\alpha_n}\Big)^r=\sum_{(\beta_1,\ldots,\beta_{r})\in \B^{r}}A_{1}^{\beta_{11}+\cdots+\beta_{r1}}\otimes \cdots \otimes A_{n}^{\beta_{1n}+\cdots+\beta_{rn}},$$
 where $\beta_{\ell}=(\beta_{\ell 1},\beta_{\ell 2},\ldots,\beta_{\ell n})\in \B$ for all $\ell=1,\ldots,r$ and thus
 \begin{equation}\label{power}
 (A^r)_{i,j}=\sum_{(\beta_1,\ldots,\beta_{r})\in \B^{r}} (A_{1}^{\beta_{11}+\cdots+\beta_{r1}}\otimes \cdots \otimes A_{n}^{\beta_{1n}+\cdots+\beta_{rn}})_{i,j}.
 \end{equation}
 Taking into account that $A$ is constructed with the lexicographic order for the vertices of $G$ which represent the ordered $n$-tuples of vertices of $G_1,\ldots,G_n$ (see \cite{CDS}). Denote by $o_{t}(i)$ the label of $\pi_{t}(v_i)$ in $G_t$  for $t=1,\ldots,n$. By definition of Kronecker product we have 
 $$(A_{1}^{\beta_{11}+\cdots+\beta_{r1}}\otimes \cdots \otimes A_{n}^{\beta_{1n}+\cdots+\beta_{rn}})_{i,j}=\prod_{t=1}^{n} (A_{t}^{\beta_{1t}+\cdots+\beta_{rt}})_{o_{t}(i),o_{t}(j)}$$
 and, by \eqref{Adjwalks} $$(A_{t}^{\beta_{1t}+\cdots+\beta_{rt}})_{o_{t}(i),o_{t}(j)}=w_{G_t}(\beta_{1t}+\cdots+\beta_{rt},\pi_{t}(v_i),\pi_{t}(v_j)).$$
Therefore, by \eqref{Adjwalks} and \eqref{power}, we obtain the desired formula.
 \end{proof}
 As a consequence we obtain a formula for the number of walks in NEPS of complete graphs.

\begin{coro}\label{walksNEPScomp}
	Let $G=NEPS(K_{m_1},\ldots,K_{m_n};\B)$. If $v_{i},v_j \in V(G)$ then
	\begin{equation}\label{numbNEPScomp}
	w_{G}(r,v_i,v_j)=\sum_{(\beta_1,\ldots,\beta_{r})\in \B^{r}}\, \prod_{t=1}^{n} a_{t}(v_i,v_j),
	\end{equation}
	where 
	$$a_{t}(v_{i},v_j)=\begin{cases}
	\frac{m_t-1}{m_t}\big((m_{t}-1)^{\beta_{1t}+\cdots+\beta_{rt}-1}-(-1)^{\beta_{1t}+\cdots+\beta_{rt}-1}\big) & \quad\text{if } \pi_{t}(v_i)=\pi_t(v_j), \\[2mm]
	\frac{1}{m_t}\big((m_{t}-1)^{\beta_{1t}+\cdots+\beta_{rt}}-(-1)^{\beta_{1t}+\cdots+\beta_{rt}}\big) & \quad \text{if } \pi_{t}(v_i)\neq\pi_t(v_j). 
	\end{cases}$$ 
\end{coro}
 \begin{proof}
 	It is enough to find $w_{K_m}(r,w_i,w_j)$, where $K_m$ is the complete graph with $m$ vertices. It is well known that
 	$$w_{K_m}(r,w_i,w_j)=
 	\begin{cases}
 		\frac{m-1}{m}\big((m-1)^{r-1}-(-1)^{r-1}\big) & \quad \text{if } w_i=w_j, \\[2mm]
 		\frac{1}{m}\big((m-1)^{r}-(-1)^{r}\big) & \quad\text{if } w_i\neq w_j.		 
 	\end{cases}$$
	Thus, the result follows from Theorem \ref{walksNEPS}.
 \end{proof}
\begin{exam}
	Let $K_3$ and $K_4$ be the complete graphs of $3$ and $4$ vertices respectively and let $G_1=\mathrm{NEPS}(K_{3}\times K_{4},\mathcal{B}_1)$ and $G_2=\mathrm{NEPS}(K_{3}\times K_{4},\mathcal{B}_2)$ with $\mathcal{B}_1=\{(1,1)\}$ and $\mathcal{B}_2=\{(1,0),(0,1)\}$. 
	
	Clearly, $\mathcal{B}_{1}^r$ only contains the element $\beta=(\beta_{1},\ldots,\beta_r)$ such that $\beta_i=\beta_j=(1,1)$ for all $i,j\in\{1,\ldots,r\}$, i.e\@ we have that $\beta_{1t}+\cdots+\beta_{rt}=r$ for $t=1,2$. On the other hand, $\mathcal{B}_{2}^{r}$ contains all the elements $\beta=(\beta_{1},\ldots,\beta_r)$ such that   
	$\beta_{1t}+\cdots+\beta_{rt}=\ell_t$ for $t=1,2$ and $\ell_i$'s satisfying $\ell_1+\ell_2=r$. 
	By Corollary \ref{walksNEPScomp} we have that  
	$$w_{G_1}(r,v,v)=(\tfrac{1}{2})(2^{r-1}-(-1)^{r-1})(3^{r-1}-(-1)^{r-1})=\tfrac{6^{r-1}+(-1)^{r}(2^{r-1}+3^{r-1})+1}{2},$$
	$$w_{G_2}(r,v,v)=\sum_{\ell=0}^{r}\tbinom{r}{\ell}(\tfrac{1}{2})(2^{\ell-1}-(-1)^{\ell-1})(3^{r-\ell-1}-(-1)^{r-\ell-1})$$
	for any vertex $v\in V(G)$.
\end{exam}

\section{Main results}

Let $\B=\{e_1,\ldots, e_n\}$, where $e_i$ is the $n$-tuple with $1$ in the position $i$ and zeros in the remainning positions. If $G$ is the NEPS  of the graphs $G_1,\ldots,G_n$ with basis $\B$, then $G=G_1+\cdots+G_n$ is the \textit{sum} of $G_t$ (cartesian product of graph). In this case we have the following result.
\begin{prop}\label{Cartform}
Let $G=G_1+\cdots+G_n$. Then, we have that	
$$w_{G}(r,v_i,v_j)=\sum_{r_1+\cdots+r_n=r} \, \frac{r!}{r_{1}!\cdots r_{n}!}\prod_{t=1}^n w_{G_t}(r_t,\pi_{t}(v_i),\pi_{t}(v_j)),$$
where $\pi_t$ denote the projection of $V(G)$ over $V(G_t)$. 
\end{prop}
\begin{proof}
Let $r$ be a non-negative integer. By Theorem \ref{walksNEPS} we have that
\begin{equation}\label{NWS}
w_{G}(r,v_i,v_j)=\sum_{(\beta_1,\ldots,\beta_{r})\in \B^{r}}\, \prod_{t=1}^{n} w_{G_t}(\beta_{1t}+\cdots+\beta_{rt},\pi_{t}(v_i),\pi_{t}(v_j)),
\end{equation}
where $\B=\{e_1,\ldots, e_n\}$.
Notice that if $(\beta_1,\ldots,\beta_r)\in \B^r$, then 
$$\beta_1+\cdots+\beta_{r}=(r_1,r_2,\ldots,r_n)\in(\mathbb{Z}_{\ge 0})^{n}\qquad \text{with} \qquad r_1+r_2+\cdots+r_n=r.$$ 
Moreover, there exist a number $\frac{r!}{r_{1}!\cdots r_{n}!}$ of $r$-tuples $(\beta_1,\ldots,\beta_r)$'s with $\beta_1+\cdots+\beta_{r}=(r_1,r_2,\ldots,r_n)$. Therefore we set \eqref{NWS} as we wanted.
\end{proof}

Recall that the Hamming graph $H(n,q)$ is the graph with vertex set all the $n$-tuples with entries from a set $\Delta$ of size $q$, and two $n$-tuples are neighbors if and only if they differ in exactly one entry. 
It is known that $H(n,q)$ is the $n$-sum of the complete graph $K_q$. Therefore, we have that
\begin{equation}\label{Walk Hamming}
w_{H(n,q)}(r,v_i,v_j)=\sum_{r_1+\cdots+r_n=r} \, \frac{r!}{r_{1}!\cdots r_{n}!}\prod_{t=1}^n a_{t}(v_i,v_j),
\end{equation}
where $r_t\ge 0$ for all $t=1,\ldots,n$ and  
\begin{equation}
a_{t}(v_{i},v_j)=\begin{cases}
\frac{q-1}{q}((q-1)^{r_t-1}-(-1)^{r_t-1}) & \quad\text{if } \pi_{t}(v_i)=\pi_t(v_j), \\[1mm]
\frac{1}{q}((q-1)^{r_t}-(-1)^{r_t}) & \quad\text{if } \pi_{t}(v_i)\neq\pi_t(v_j). 
\end{cases}
\end{equation}

In \cite{LP}, the authors caracterized all  generalized Paley graphs which are Hamming graphs. More precisely, they showed that $\G(k,p^m)$ is a Hamming graph if and only if $u=b(p^{ a}-1)$ for some divisor $b>1$ such that $m=ab$. 

Also if $\omega$ is a primitive element of $\ff_{p^m}$, then the set $\{1,\,\omega^k,\,\omega^{2k},\ldots,\,\omega^{(b-1)k}\}$ is a basis of $\ff_{p^m}$ as $\ff_{p^{a}}$-vector space, then $$x=\sum_{i=0}^{b-1} c_i \,\omega^{i\,k}\mapsto [x]=(c_0,\,c_1,\ldots,\,c_{b-1})\in (\ff_{p^{a}})^b$$ and hence we have the isomorphism (see \cite{LP})
\begin{equation}\label{iso GP=HAM}
\G(k,p^m)\cong H(b, p^{a}).
\end{equation}

In this case, we have the following result.
\begin{prop}\label{propwalks}
	Let $p$ be a prime and let $m,k$ be positive integers such that $k\mid p^m-1$. 
	If $u=\frac{p^m-1}{k}=b(p^{a}-1)$ with $b>1$ and $m=ab$, then  
	\begin{equation}
	w_{\G(k,p^m)}(r,x,y)=w_{H(b,p^{a})}(r,[x],[y])=\sum_{r_1+\cdots+r_b=r} \, \frac{r!}{r_{1}!\cdots r_{b}!}\prod_{i=1}^b a_{i}(x,y),
	\end{equation} where $r_i\ge 0$ for all $i=1,\ldots,b$ and  
	$$a_{i}(x,y)=\begin{cases}
	\frac{p^{a}-1}{p^{a}}((p^{a}-1)^{r_i-1}-(-1)^{r_i-1}) & \quad\text{if } [x]_i=[y]_i, \\[1mm]
	\frac{1}{p^{a}}((p^{a}-1)^{r_i}-(-1)^{r_i}) & \quad\text{if } [x]_i\neq[y]_i. 
	\end{cases}$$ 
	where $[x]_i$ denotes the $i$-th coordinate of the vector $[x]\in (\ff_{p^{a}})^b$, respect to the $\ff_{p^{a}}$-base  $\{1,\,\omega^k,\,\omega^{2k},\ldots,\,\omega^{(b-1)k}\}$.
\end{prop}

As a direct consequence of the Theorem \ref{walksolthm} and Proposition \ref{propwalks} we obtain the following Proposition.
\begin{prop}\label{coromaster}
	Let $p$ be a prime and let $a,b$ be positive integers such that $b>1$. If $k=\frac{p^{ab}-1}{b(p^a-1)}$ is integer 
	and $N_r$ denotes of solutions $(x_1,\ldots,x_r)$ in $(\ff_{p^{ab}}^*)^r$ to the diagonal equation 
	$x_1^{k}+\cdots+x_{r}^k= \alpha$, then  
	\begin{equation}\label{N}
	N_r=k^{r}\sum_{r_1+\cdots+r_b=r} \, \frac{r!}{r_{1}!\cdots r_{b}!}\prod_{i=1}^b a_{i}(\alpha),
	\end{equation} 
	where $r_i\ge 0$ for all $i=1,\ldots,b$ and  
	$$a_{i}(\alpha)=
	\begin{cases}
	\frac{p^{a}-1}{p^{a}}((p^{a}-1)^{r_i-1}-(-1)^{r_i-1}) & \quad\text{if } [\alpha]_i=0, \\[1mm]
	\frac{1}{p^{a}}((p^{a}-1)^{r_i}-(-1)^{r_i}) & \quad\text{if } [\alpha]_i\neq 0.
	\end{cases}$$
	where $[\alpha]_i$ denotes the $i$-th coordinate of the vector $[\alpha]\in (\ff_{p^{a}})^b$, respect to the $\ff_{p^{a}}$-base  $\{1,\,\omega^k,\,\omega^{2k},\ldots,\,\omega^{(b-1)k}\}$ of $\ff_{p^{ab}}$.
\end{prop}

As a direct consequence of the last proposition and Remark \ref{rem imp} we obtain our main result.
\begin{thm}
Let $p$ be a prime and let $a,b$ be positive integers such that $b>1$. If $k=\frac{p^{ab}-1}{b(p^a-1)}$ is integer 
and $M_s$ denotes of solutions $(x_1,\ldots,x_s)$ in $(\ff_{p^{ab}})^s$ to the diagonal equation 
$x_1^{k}+\cdots+x_{s}^k= \alpha$, then 
\begin{equation}
M_s=
\begin{cases}
\sum_{\ell=1}^r\tbinom{s}{r}N_{r} \qquad \text{if } \alpha\ne 0,\\
1+\sum_{i=1}^r\tbinom{s}{r}N_{r} \quad \text{if } \alpha= 0, 
\end{cases}
\end{equation}
where $N_r$ is given by \eqref{N}.
\end{thm}
\begin{rem}
	Clearly, the hypothesis $k$ to be integer is equivalent to $b\mid \frac{p^{ab}-1}{p^a-1}$. 
	This condition was recently studied in \cite{PV}. 
	More specifically, if we put $x=p^a$, then $b\mid \frac{p^{ab}-1}{p^a-1}$ in the following cases:
	\begin{enumerate}[$(a)$]
		\item If $b=r$ is a prime different from $p$ 
		and $x\equiv 1 \pmod r$.
		
		\item If $b=2r$ with $r$ an odd prime,  $x$ coprime with $b$ and $x\equiv \pm1 \pmod r$.
		
		\item If $b=r r'$ with $r<r'$ odd primes such that $r \nmid r'-1$ and $x\equiv 1 \pmod{rr'}$.
		
		\item If $b=r_1 r_2 \cdots r_\ell$ with $r_1 < r_2 < \cdots < r_\ell$ primes different from $p$ with $x\equiv 1 \pmod{r_1}$ and $x^{b/r_i} \equiv 1 \pmod{r_i}$ for $i=2, \ldots,\ell$.
		
		\item If $b=r^t$ with $r$ prime such that $ord_{b}(x)=r^h$ for some $0\le h<t$.
		
		\item If $b = r_1^{t_1}	\cdots r_\ell^{t_\ell}$ with $r_1 < \cdots < r_\ell$ primes different from $p$ where $ord_{r_{i}^{t_i}}(x)=r_{i}^{h_i}$ with  $0\le h_i\le t_{i}-1$ for all $i$.
	\end{enumerate}
\end{rem}

\end{document}